\documentclass[16pt]{amsart}
\usepackage{amsmath,amsfonts,amssymb,amsthm,amscd,latexsym,euscript,xypic, verbatim}
\usepackage{epsf}
\usepackage{epsfig}
\usepackage{amstext}
\usepackage{xspace}
\usepackage{amsfonts}
\usepackage{amsmath}
\usepackage{amssymb}
\usepackage{xspace}
\usepackage{graphicx}
\usepackage[active]{srcltx}
\usepackage[all]{xy}
\usepackage{graphics}

\newcounter{tempcounter}

\swapnumbers
\theoremstyle{plain}

\newtheorem*{thm1.2}{(1.2) Theorem}
\newtheorem*{thm1.3}{(1.3) Theorem}
\newtheorem*{thm1.4}{(1.4) Theorem}
\newtheorem*{propA*}{Proposition A}
\newtheorem*{propB*}{Proposition B}
\newtheorem*{thmC*}{Theorem C}
\newtheorem*{propD*}{Proposition D}
\newtheorem{prop}{Proposition}[section]
\newtheorem{thm}[prop]{Theorem}
\newtheorem{cor}[prop]{Corollary}
\newtheorem{lemma}[prop]{Lemma}

\theoremstyle{definition}

\newtheorem{Def}[prop]{Definition}
\newtheorem*{Def*}{Definition}

\newtheorem{example}[prop]{Example}

\newtheorem*{notation*}{Notation}
\newtheorem*{question*}{Question}

\newcommand{\pspaces}{\text{\rm Spaces}_{\ast}}
\newcommand{\calc}{\mathcal C}
\newcommand{\calk}{\mathcal K}

\newcommand{\kk}{\mathcal K}
\newcommand{\z}{\mathbb Z}

\newcommand{\ra}{\rightarrow}

\xyoption{2cell}

\title[Homotopy colimits of nilpotent spaces]{Homotopy colimits of nilpotent spaces}

\author[W. Chach\'olski]{Wojciech Chach\'olski}
\thanks{The first author was partially supported by G\"oran Gustafsson Stiftelse and VR grants.}
\author[E. Dror Farjoun]{Emmanuel Dror Farjoun}
\author[R. Flores]{Ram\'on Flores}
\thanks{The third and fourth authors are partially supported by MTM2010-20622, and UNAB10-4E-378 ``Una manera de hacer Europa".}
\author[J. Scherer]{J\'er\^ome Scherer}

\address{Department of Mathematics\\ KTH Stockholm
Lindstedtsv\" agen 25 \\ 10044 Stockholm \\ Sweden}
\email{wojtek@math.kth.se}

\address{Department of Mathematics\\ Hebrew University of Jerusalem\\
Givat Ram\\Jerusalem 91904\\ Israel}
\email{farjoun@math.huji.ac.il}

\address{Department of Statistics\\ Universidad Carlos III de Madrid\\
Avda. de la Universidad Carlos III, 22 \\ 28270 Colmenarejo (Madrid) \\ Spain}
\email{rflores@est-econ.uc3m.es }

\address{Department of Mathematics\\ EPFL Lausanne \\
Station 8 \\1015 Lausanne \\ Switzerland }
\email{jerome.scherer@epfl.ch}


\begin{document}

\begin{abstract}
We show that cellular approximations of nilpotent Postnikov stages are
always nilpotent Postnikov stages, in particular classifying spaces of nilpotent groups
are turned into classifying spaces of nilpotent groups. We use a modified Bousfield-Kan
homology completion tower $z_k X$ whose terms we prove are all $X$-cellular for any $X$. 
As straightforward consequences, we show that if $X$ is $\mathcal K$-acyclic 
and nilpotent for a given homology theory $\mathcal K$, then so are all its Postnikov sections $P_nX$,
and that any nilpotent space for which the space of pointed self-maps
$\text{map}_*(X,X)$ is ``canonically" discrete must be aspherical.
\end{abstract}

\date{\today}

\maketitle

\section{Introduction}
The Postnikov tower of a nilpotent space $X$ is considered classically as a mean to reconstruct its Postnikov sections
$P_n X$ out of  Eilenberg--Mac Lane spaces (basic homotopical  building blocks)  by principal
fibration sequences (see Example~\ref{ex:Postnikov} for what we mean by  Postnikov sections). 
Eventually the space $X$ itself is recovered as homotopy limit of the Postnikov tower.

In this work we change the perspective and show that, if $X$ is  nilpotent, then its Postnikov sections can  
be constructed out  of $X$ by means of wedges, homotopy push-outs, and telescopes
(see Theorem~\ref{thm main}). This is certainly not true for arbitrary spaces
and, when $X$ is nilpotent, it allows us to deduce strong  properties of its  Postnikov sections.
Let us first mention a few of these consequences  and then present our methods and techniques. We provide a topological 
strengthening of the classical Serre class statements on the relation between homotopy and homology groups, see \cite{MR0059548}.

\medskip

\noindent
\ref{thm homology}.~{\bf Theorem.}\
\emph{
Let $\kk$ be a reduced homology theory.
\begin{enumerate}
\item Assume $X$ is $\kk$-acyclic. If $P_nX$ is nilpotent, then it is also $\kk$-acyclic.
\item Assume  $X$ is nilpotent. Then $\prod_{k \geq 1} K(\pi_kX ,k)$ is $\kk$-acyclic if and only if
$\prod_{k \geq 1} K(H_k(X,{\mathbb Z}),k)$ is  $\kk$-acyclic.
\item If $K(G,1)$ is $\kk$-acyclic, then so is $K(G/\Gamma_n G,1)$ for any $n$.
\end{enumerate}
}

\medskip

In~\cite{MR95c:55010} and \cite[Lemma 5.3]{MR98m:55009}  Bousfield proved his celebrated result known
as the ``key lemma". All his  results   related to understanding the  failure  of preservation of fibrations by localizations
depended on it. The key lemma implies for example that,  if $X$ Êis simply connected, then  the map $\pi_n\colon\text{map}_{\ast}(X,X)\to
\text{Hom}(\pi_nX,\pi_nX)$ is a weak equivalence if and only if $X$Ê is weakly equivalent to
$K(\pi_nX,n)$. If  $X$ Êis not simply connected, then this is far from being true. In this article we offer
an extension of this last result to spaces with a nilpotent fundamental group.

\medskip
\noindent
\ref{cor maindiscmapsp}.~{\bf Theorem.}\
\emph{Let $X$ be a connected space whose fundamental group  $\pi_1X$ is nilpotent. Assume that
the map $\pi_1\colon\text{map}_{\ast}(X,X)\to\text{Hom}(\pi_1X,\pi_1X)$ is a weak equivalence.
Then $X$ is weakly equivalent to $K(\pi_1X, 1)$.
}
\medskip

For arbitrary fundamental groups, this fails as illustrated in~\cite[Example 2.6]{MR2351607}
by a space $X$ with $\pi_1 X \cong \Sigma_3$ whose universal cover is the homotopy fiber of the degree $3$
map on the sphere $S^3$. We  come back to this space in Example~\ref{ex symmetric}.

\medskip

To prove the above results we use cellularization techniques.
Looking at spaces through the eyes of a given space $A$ via the pointed mapping space $\text{\rm map}_{\ast}(A, - )$
is the central idea (\cite{MR98f:55010}, \cite{MR97i:55023}). Recently cellularization has found  applications in other contexts:
Dwyer, Greenlees, and Iyengar used it to investigate duality in stable homotopy, \cite{MR2200850}, see also \cite{MR2520402};
periodicity phenomena in unstable homotopy theory are the subject of \cite{MR2242617}; cellularity has also been intensively 
studied in group theory, see for example \cite{MR2269828} for an early general reference, and \cite{MR2995665} for a recent
point of view on finite (simple) groups; in \cite{MR2581213} Kiessling computed explicitly the cellular lattice of certain perfect 
chain complexes, refining the Bousfield lattice.

Classically $A$ is the zero-sphere $S^0$ and
we are doing standard homotopy theory. In this context the cellularization of a space is nothing else than the cellular approximation of
the space, the best approximation built out of $S^0$ and its suspensions~$S^n$ via pointed homotopy colimits.
Changing the sphere for another space $A$ allows us to modify our point of view and we do now $A$-homotopy theory.
The cellularization functor $\text{cell}_A: \text{Spaces}_{\ast}\to \text{Spaces}_{\ast}$ on pointed
spaces retains from a space its essence from the point of view of $A$.
A space of the form $\text{cell}_A Y$ is $A$-cellular in the sense that it can be constructed from $A$ and its suspensions
by pointed homotopy colimits. When a space $X$ is $A$-cellular we write $X \gg A$.

It turns out that to prove the above theorems we need to show that $\text{cell}_A$
behaves nicely on nilpotent Postnikov sections of any space.
A  result of the first author \cite[Section 7]{MR97i:55023} says that
the structure map $\text{cell}_A X \to X$ is a principal fibration. Hence, the cellularization of a nilpotent space is always nilpotent,
which hints at the tame behavior of $\text{cell}_A$ on nilpotent spaces. We show
in particular in Corollary~\ref{cor npostpreserv} that if $X$ is a nilpotent $n$-Postnikov  stage
(see the end of Section~\ref{sec:notation}), then so is $\text{cell}_A X$.
When $n=2$, this result is in perfect accord with a purely group-theoretical result, \cite[Theorem~1.4 (1)]{MR2269828},
showing that in the category of groups any cellularization of a nilpotent group is nilpotent.
To understand how cellularization functors affect Postnikov sections, we rely on the following statement which we regard as the main result of this paper.

\medskip
\noindent
\ref{thm main}.~{\bf Theorem.}\
\emph{ Let $X$ be a space. If  $P_nX$
 is nilpotent, then $P_nX\gg X$, namely any Postnikov section is $X$-cellular.
}

\medskip

This theorem is a consequence of  preservation of polyGEM by cellularization functors.
A space is called a $1$-polyGEM if it is a GEM, i.e.\  a product of Eilenberg-MacLane spaces $K(A_i,i)$ for abelian  groups~$A_i$.
It is called an $n$-polyGEM, for $n>1$, if it is weakly equivalent to a retract of the homotopy fiber of a map $f\colon X\to Y$
where $X$ is $(n-1)$-polyGEM and $Y$ is a GEM. We prove in Theorem~\ref{thm preservationpolygem} that cellularization
functors turn $n$-polyGEMs into $n$-polyGEMs.

To prove the preservation of polyGEMs we need detection tools for polyGEMs which behave well with respect to cellularity.
It is well known that a space $X$ is a GEM  if and only if it is a retract
 of $\z X,$ where $\z X$
 is the Dold-Thom free abelian group construction on $X$, see for example \cite[Definition~3.5]{MR0279808}, the simplicial version of
 the infinite symmetric product $SP^\infty X$ (see also \cite{MR0097062}). We need similar functors
 that detect the property of being a polyGEM.
 As a first attempt one might try to consider the functors $\z_k X$
 in  the Bousfield-Kan completion tower with respect to the integral homology
\cite{MR51:1825}. Although the spaces in this tower are polyGEMs, they are very special polyGEMs, so called flat~\cite{MR1953240}. Thus instead of general polyGEMs, the functors $\z_k X$  detect only flat polyGEMs. Furthermore we do not  know
the needed cellularity properties of  $\z_k X$.

It turns out that for our purposes,
instead of the classical  Bousfield-Kan completion tower, one should  consider
 the modified Bousfield-Kan tower $\{ z_k X \}$ constructed by the second author in \cite{MR2002e:55012}. We do that
in Section~\ref{sec  injtwist}.
This tower has two key properties that are essential in our work. First, it detects polyGEMs.

\medskip
\noindent
\ref{polyretract}.~{\bf Proposition.}\
\emph{A space $W$ is a polyGEM if and only if it is a retract of $z_n W$ for  some~$n$.
}

\medskip
Second, the functors in the modified tower are cellular, a property that we are
unable to prove for the functors in the classical  Bousfield-Kan tower.

\medskip
\noindent
\ref{prop twistB-cellular}.~{\bf Proposition.}\
\emph{For all $k\geq 0 $  the co-augmented functor  $z_k$ is cellular. In particular $z_k X$ is $X$-cellular for any $X$.
}

\medskip

We define and study cellular functors in Section~\ref{sec Bousfield}. The functor $z_k$ being cellular
tells us much more than the fact that $z_k X$ is $X$-cellular. It says that the homotopy cofiber of the augmentation $X \to z_k X$
is $\Sigma X$-acyclic, so that $z_k X$ can be built from $X$ starting from $X$ and adding higher cells $\Sigma^i X$
for $i \geq 1$. This is how Bousfield's original ideas about his key lemma come into play.

\medskip

\noindent
{\bf Acknowledgements.} We would like to thank Bill Dwyer for helpful discussions and suggestions.




\section{Notation and set-up}
\label{sec:notation}
In this section we will recall basic definitions related to cellularity and acyclicity of spaces and state their fundamental properties. For more information about these notions we refer the reader to~\cite{MR98f:55010}.

The category of pointed simplicial sets with the standard simplicial model structure is denoted by $\text{Spaces}_{\ast}$.
Its objects are called pointed spaces or simply spaces, and morphisms are called maps. The space of maps between two pointed spaces $X$ and $Y$  is denoted by $\text{map}_{\ast}(X,Y)$.

 We say that a class  $\calc$ of pointed spaces is {closed under weak equivalences} if
when  $X$ belongs to $\calc$, then so does any pointed space  weakly equivalent to $X$.
We say that $\calc$ is closed under homotopy colimits if, for any functor $F\colon I\to \text{Spaces}_{\ast}$ whose values belong to
$\calc$, the  homotopy colimit $\text{hocolim}_{I} F$ in  $\pspaces$ also belongs to~$\calc$.
A class  of pointed spaces $\calc$ which  is closed both under weak equivalences  and homotopy colimits
is called {\bf cellular}.

Cellular classes were called
closed classes in \cite{MR98f:55010}, but we believe that the name ``cellular" is more descriptive in our context.
For a class $\calc$ to be cellular it is sufficient if it is   closed under weak equivalences,
arbitrary wedges and homotopy push-outs. This is so since all pointed homotopy colimits can be built by repeatedly
using these two special cases. Furthermore a  { retract} of a member of a cellular class also belongs to the cellular class, \cite[2.D.1.5]{MR98f:55010}, where
recall that $X$ is a  retract of $Y$ if there are maps
$f\colon X\to Y$ and $r\colon Y\to X'$ whose composition $rf$ is a weak equivalence.

The symbol $\calc(A)$ denotes the smallest cellular class in $\text{Spaces}_{\ast}$ containing a given  space $A$. If $X$
belongs to $\calc(A)$, then we write  $X\gg A$ and say that $X$ is $A$-cellular or that $A$ builds $X$.
For example, $\calc(S^0)$  consists of all pointed spaces and $\calc(S^n)$ of all $(n-1)$-connected pointed spaces.

A weaker notion than cellularity is given by acyclicity (see \cite{MR97i:55023} for the origin of the terminology). For a
map $f \colon X \to Y$ of pointed spaces,  $\text{Fib}(f)$ denotes the homotopy fiber of $f$ over the base point.
A cellular class $\calc$ is called  {\bf acyclic}  if, for any map $f\colon X\to Y$ such that $Y$ and $\text{Fib}(f)$
belong to $\calc$,  the space   $X$   belongs to~$\calc$.
We also say that $\calc$ is {closed under extensions by fibrations}.

Given a pointed  space $A$, the symbol $\overline{\calc(A)}$
denotes the smallest acyclic class in  $\text{Spaces}_{\ast}$ containing~$A$. If a space $X$ belongs to $\overline{\calc(A)}$, then we write  $X> A$ and say
that $X$ is $A$-acyclic. There is an obvious inclusion $\calc(A)\subset \overline{\calc(A)}$ which in general is proper.
For example, if $p$ is a prime number and $G$  a finite $p$-group, $K(G,1)$ is always $K({\z}/p,1)$-acyclic; however,
$K(G,1)$ is $K({\z}/p,1)$-cellular if and only if $G$ is generated by elements of order $p$, see \cite[Section 4]{Ramon} for details.

In \cite{MR97i:55023}, the first author proved that cellularity can be detected by means of a universal property:

\begin{thm} \label{thm universal cell}
A pointed space $X$ is $A$-cellular if and only if, for any  map $f$  between fibrant spaces such that $\text{\rm map}_{\ast}(A,f)$
is a weak equivalence, then $\text{\rm map}_{\ast}(X,f)$ is also a weak equivalence.
\end{thm}

The present paper deals with possible values of the  $A$-cellular approximation
or $A$-cellular cover of spaces. The existence and
basic properties of these cellular cover functors $\text{cell}_A$ are guaranteed by the following result proved in
\cite[Section 2]{MR98f:55010}, see also \cite{MR97i:55023}.

\begin{thm}\label{thm excellul}
Let $A$ be a pointed space.
 \begin{enumerate}
   \item[(A)] There is a   natural fibration  $c_{A,X}\colon\text{\rm cell}_AX\twoheadrightarrow X$ in $\text{\rm Spaces}_{\ast}$
  such that:
  \begin{itemize}
       \item $\text{\rm cell}_A$ preserves  weak equivalences;
       \item$ \text{\rm  cell}_AX$ is $A$-cellular;
       \item  the map $\text{\rm map}_{\ast}(A,c_{A,X})$ is a weak equivalence.
   \end{itemize}
  \item[(B)] A pointed space $X$ is $A$-cellular if and only if the map  $c_{A,X}\colon\text{\rm cell}_AX\twoheadrightarrow X$  is a weak equivalence.
\end{enumerate}
\end{thm}
The map $c_{A,X}\colon\text{\rm cell}_AX\twoheadrightarrow X$, given by Theorem~\ref{thm excellul},   is called the {\bf $A$-cellular cover} of $X$  and the functor
$\text{\rm cell}_A\colon \text{\rm Spaces}_\ast\to  \text{\rm Spaces}_\ast$ the $A$-{\bf cellularization}.

\begin{example}\label{ex:Postnikov}
Let $S^{n+1}$be an $(n+1)$-dimensional sphere. The $S^{n+1}$-cellular cover $c_{S^{n+1},X}\colon \text{\rm cell}_{S^{n+1}}X\to X$
coincides with the $n$-connected cover and fits into a fibration sequence:
\[
\text{\rm cell}_{S^{n+1}}X\xrightarrow{c_{S^{n+1},X}} X\xrightarrow{p_{n,X}} P_nX
\]
where $p_{n,X}\colon X\to P_nX$ is the $n$-th {\bf Postnikov section}.
We call a  space $X$ an $n$-{\bf Postnikov stage} if  the map $p_{n,X}\colon X\to P_nX$ is a  weak equivalence,
that is, if $\pi_iX=0$ for $i\geq n+1$. Thus a $0$-Postnikov stage is homotopically discrete,
and a connected $1$-Postnikov stage is an Eilenberg--Mac Lane space $K(\pi,1)$.
\end{example}

Similarly to the $n$-connected cover, for any $A$, the map $c_{A,X}\colon\text{\rm cell}_AX\twoheadrightarrow X$ is always a principal fibration, \cite[Corollary~20.7]{MR97i:55023}.

\section{Basic cellular inequalities}
\label{sec basic}
This section contains fundamental cellular and acyclic inequalities.
The idea is to use these inequalities as basic moves to obtain more involved ones.
Although  our basic result is an inequality $P_n X\gg X$ for a nilpotent  $P_nX,$ we will use on the way
many general inequalities that hold without any nilpotency condition.
The following results describe our basic dictionary of cellular inequalities, which in most cases admit direct proofs.

Our first statements, formulated for cellular inequalities $\gg$, hold also as stated for the weak inequality $>$.

\begin{prop}\label{prop basicmoves1}\hspace{1mm}
Let $A$ and $X$ be pointed spaces.

\begin{enumerate}
\item
 If  $X\gg A$, then,  for any space $E$,  $E\wedge X\gg  E\wedge A$.
 \item If $X$ is connected, then $\Omega X\gg  A$ if and only if $X\gg  S^1\wedge A$.
\item If  $X\gg A$ and $A$ is connected, then $\Omega X\gg  \Omega A$.
\item For any $X$ and $n\geq 1$, $\Omega^n(S^n\wedge X)\gg X$.
     \setcounter{tempcounter}{\value{enumi}}
     \end{enumerate}
\end{prop}

\begin{proof}
The first cellular inequality is the content of~\cite[Theorem 4.3]{MR98e:55014}, while the acyclic inequality follows from the cellular one
and the fact that, for any space $A$, there is a space $B$ such that $\overline{{\calc}(A)}=\calc(B)$, see~\cite[Corollary 6.2]{MR2073317}.
The cellular inequality (2) is proved  in~\cite[Theorem 10.8]{MR97i:55023} and the acyclic one is~\cite[Theorem 18.5]{MR97i:55023}.
The inequalities (3) and (4) are easy consequences of (1) and  (2).
\end{proof}

Our second set of inequalities concerns homotopy fibers and cofibers of a map $f \colon X \to Y$. The last one tells us that
the ``fiber of the cofiber" is close to the space $X$ from a cellular point of view.

\begin{prop}\label{prop basicmoves2}
Let   $f\colon X\to Y$ be a map to a connected pointed space $Y$.
\begin{enumerate}
\setcounter{enumi}{\value{tempcounter}}
\item  $\text{\rm Cof}(f)\gg S^1\wedge \text{\rm Fib}(f)$.
\item For any $E$,   $\text{\rm Fib}(E\wedge f)\gg E\wedge \text{\rm Fib}(f)$.
 \item If $\alpha\colon Y\to\text{\rm Cof}(f)$ is a homotopy cofiber of $f$, then $\text{\rm Fib}(\alpha)\gg X$.
     \setcounter{tempcounter}{\value{enumi}}
 \setcounter{tempcounter}{\value{enumi}}
\end{enumerate}
\end{prop}

\begin{proof}
Inequality (5) is~\cite[Proposition 10.5]{MR97i:55023} while
(6) follows from the case of a circle \cite{MR96g:55011}, and induction on a cell decomposition.
Finally (7) is~\cite[Corollary~9.A.10]{MR98f:55010} (see also~\cite[Proposition 4.5.(4)]{MR97i:55023}).
\end{proof}

In fibration and cofibration sequences we can sometimes relate the cellularity of the spaces involved in the sequence.
The main difficulty for a fibration sequence $F\to E\to B$ is that in general we cannot
extract information about the total space of the type $E\gg A$ knowing that $B,F\gg A$.

\begin{prop}\label{prop basicmoves3}
Let   $Z\to X\to Y$ be either a cofibration or a fibration sequence.
\begin{enumerate}
\setcounter{enumi}{\value{tempcounter}}
\item If $Z>A$ and $Y>A$, then $X>A$.
\item If $Z\gg A$ and $Y> S^1\wedge A$, then $X\gg A$.
 \item  If $Z> A$ and $X\gg A$, then $Y\gg A$.
   \setcounter{tempcounter}{\value{enumi}}
 \end{enumerate}
\end{prop}

\begin{proof}
For a fibration sequence inequality (8) holds by definition of the relation~$>$. The cofibration sequence case follows from (7)
and the fibration sequence case. Inequality (9) is~\cite[Corollary 20.2]{MR97i:55023} for a fibration sequence, and the cofibration
case follows again from (7) and the fibration sequence case. Finally, inequality (10) may be deduced from statements (5) and (9).
\end{proof}

The following inequalities, that put in relation fiber and cofiber from the cellular point of view, will be especially relevant in the rest of the paper.

\begin{prop}\label{prop basicmoves4} \hspace{1mm}
\begin{enumerate}
\setcounter{enumi}{\value{tempcounter}}
    \item Let $X$ be a connected pointed space and $n \geq 1$. Let  $e_{n,X}\colon X\to\Omega^n(S^n\wedge X)$ be the adjoint to
  $\text{\rm id}\colon S^n\wedge X\to S^n\wedge X$.
  Then $\text{\rm Fib}(e_{n, X})\gg S^1\wedge \Omega X\wedge \Omega X$ and $\text{\rm Cof}(e_{n, X})\gg X\wedge X$.
   \setcounter{tempcounter}{\value{enumi}}
 \end{enumerate}
\end{prop}

\begin{proof}
If  $n=1$, then the statement for the homotopy fiber is~\cite[Theorem 7.2]{MR98i:55013}.
We proceed by induction on $n$. Let $n>1$. The map $e_{n, X}$ factors as the composition:
\[
\xymatrix{
X \ar@/_17pt/[rrrrr]_{e_{n, X}} \rrto^-{e_{n-1,X}} &&  \Omega^{n-1}(S^{n-1}\wedge X) \ar[rrr]^-{\Omega^{n-1}(e_{1,S^{n-1}\wedge X})} & &&
\Omega^{n}(S^{n}\wedge X)
}\]
which  leads to a fibration  sequence:
\[
\text{Fib}(e_{n-1,X})\ra \text{Fib}(e_{n,X})\ra \text{Fib}(\Omega^{n-1}(e_{1,S^{n-1}\wedge X}))
\]
and a cofibration sequence:
\[
\text{Cof}(e_{n-1,X})\ra \text{Cof}(e_{n,X})\ra \text{Cof}(\Omega^{n-1}(e_{1,S^{n-1}\wedge X}))
\]
Let us analyze $\text{Fib}\left(\Omega^{n-1}(e_{1,S^{n-1}\wedge X})\right) \simeq
\Omega^{n-1}\left(\text{Fib}(e_{1,S^{n-1}\wedge X})\right)$. According to   statements (2) and (3) the following
cellular relation holds:
\[
\Omega^{n-1}\left(S^1\wedge \Omega(S^{n-1}\wedge X)\wedge   \Omega(S^{n-1}\wedge X)\right)\gg
  \Omega^{n-1}\left( S^1\wedge S^{n-2}\wedge X\wedge S^{n-2}\wedge X  \right)
\]
 For this last space   $\Omega^{n-1}\left( S^{2n-3}\wedge X\wedge X\right)$,  by the same
argument, we have the following relation:
\[
\Omega^{n-1}\left( S^{2n-3}\wedge X\wedge X\right)\gg S^{n-2}\wedge X\wedge X\gg S^n\wedge  \Omega X\wedge  \Omega X\gg S^2\wedge \Omega X\wedge \Omega X
 \]
Thus $\text{\rm Fib}(e_{n, X})\gg S^1\wedge \Omega X\wedge \Omega X$ is a consequence of the inductive step and statement  (9). The case of the homotopy cofiber is an analogous induction based
on the proof of~\cite[Theorem 7.2]{MR98i:55013}, which analyzes the James construction.
\end{proof}

We arrive now at more delicate inequalities involving diagrams. Consider a homotopy push-out square:
\[\xymatrix{
A\rto^{f}\dto_g & B\dto^{h}\\
C\rto^k  & D
}\]
Let  $\beta\colon \text{\rm Fib}(g)\to \text{\rm Fib}(h)$
and $\gamma\colon A\to T:=\text{holim}(C\xrightarrow{k} D\xleftarrow{h} B)$
be the maps induced by the commutativity of this square.

\begin{thm}
\label{thm basicmoves}\hspace{1mm}
\begin{enumerate}
\setcounter{enumi}{\value{tempcounter}}
\item If $C$ and $D$ are connected, then $\text{\rm Fib}(h)\gg \text{\rm Fib}(g)$.
\item If $\text{\rm Fib}(k)$ is connected, then $\text{\rm Cof}(\beta)\gg S^1\wedge\text{\rm Fib}(g)\wedge \Omega \text{\rm Fib}(k)$.
\item  If $B$, $C$, $T$ and  $\text{\rm Fib}(\gamma)$ are connected, then
$\text{\rm Fib}(\gamma)>S^1\wedge \Omega \text{\rm Fib}(f)\wedge  \Omega \text{\rm Fib}(g)$.
     \setcounter{tempcounter}{\value{enumi}}
\end{enumerate}
\end{thm}

\begin{proof}

Inequality (12) is~\cite[Theorem 3.4]{MR98e:55014}. To check
(13), observe that according to Puppe's theorem we can form the following homotopy-push out square
 of homotopy fibers:
\[\xymatrix{
\text{Fib}(g)\drto^{h'}\ar@/^17pt/[drr]^{\beta}\\
&\text{Fib}(hf)\rto^{f'}\dto_{g'} & \text{Fib}(h)\dto\\
&\text{Fib}(k)\rto &\Delta[0]
}\]
where  the  map $h'\colon\text{Fib}(g)\to \text{Fib}(hf)$ is a homotopy fiber of $g'\colon\text{Fib}(hf)\to\text{Fib}(k)$.
Let $\text{Fib}(k)\hookrightarrow R\text{Fib}(k)$ be a weak equivalence into a fibrant space.
Since the composition $g'h'$ factors through a contractible space,
by taking the cofibers of $h'$ and $\beta$, we can form a new homotopy push-out square:
\[\xymatrix{
\text{Cof}(h')\dto_{g''} \rto & \text{Cof}(\beta)\dto\\
R\text{Fib}(k)\rto & \Delta[0]
}\]
By Ganea's theorem~\cite{MR31:4033}, $\text{Fib}(g'')\simeq S^1\wedge \text{Fib}(g)\wedge\Omega \text{Fib}(k)$.
Under the assumption that $\text{Fib}(k)$ is connected, we can then apply statement (13)
to this last  homotopy push-out square to get $\text{Cof}(\beta)\gg S^1\wedge \text{Fib}(g)\wedge\Omega \text{Fib}(k)$.
Finally, inequality (14) is~\cite[Theorem 5.1]{newblackersmassey}.
\end{proof}

\section{Cellular functors and Bousfield Key Lemma}
\label{sec Bousfield}
An essential tool to understand the failure of preservation of fibrations by localizations and cellularizations is the so-called ``key lemma"
of  Bousfield~\cite[Theorem 5.3]{MR98m:55009} and Dror Farjoun~\cite{MR98f:55010}. In its original form it states that if, for connected
pointed spaces $X$ and $Y$, $\text{map}_{\ast}(X,Y)$ is weakly equivalent to a discrete space and $\pi_1Y$ acts trivially on the set of
components $\pi_0\text{map}_{\ast}(X,Y)$, then the Hurewicz map $ h_X\colon X \to {\mathbb Z} X$ induces a weak equivalence between
$\text{map}_{\ast}(X,Y)$ and $\text{map}_{\ast}({\mathbb Z} X,Y)$. The space $\mathbb Z X$ is the simplicial abelian group freely generated
by the simplices of $X$. It is the simplicial model for the Dold-Thom infinite symmetric product $SP^\infty X \simeq \prod K(H_n (X; \mathbb Z) , n)$,
see for example \cite[3.11]{MR0279808}.

The assumptions of the key lemma are equivalent to the statement that
the inclusion $S^1\vee X\hookrightarrow  S^1\times X$ induces a weak equivalence between
$\text{map}_{\ast}(S^1\vee X, Y)$ and $\text{map}_{\ast}(S^1\times X, Y)$
which means that $Y$ is  local with respect to  $S^1\vee X\hookrightarrow  S^1\times X$.
The key lemma states therefore that, for a connected space $X$,  the Hurewicz map $h_X\colon X \to {\mathbb Z} X$ is an
$L_{S^1\vee X\hookrightarrow  S^1\times X}$-equivalence
(see~\cite{MR98f:55010}). One way to prove this lemma is to show a stronger statement:
 the map $h_X\colon X \to {\mathbb Z} X$ is constructed inductively starting with  $X$ and taking push-outs
along the inclusion $S^1\vee X\hookrightarrow  S^1\times X$ or its suspensions.
Since the cofiber of the map $S^1\vee X\hookrightarrow  S^1\times X$  is
$S^1\wedge X$, it follows that $\text{Cof}(h_X)> S^1\wedge X$
(in fact a stronger relation  $\text{Cof}(h_X)\gg S^1\wedge X$ holds in this case)~\cite[Section~4.A]{MR98f:55010}.
Furthermore if $X$ is simply connected, the same argument can be used to show that the
Hurewicz map $h_X\colon X \to {\mathbb Z} X$  is constructed inductively starting with  $X$ and taking push-outs along
the inclusion $S^2\vee X\hookrightarrow  S^2\times X$ or its suspensions. By the Ganea theorem~\cite{MR31:4033},
$\text{Fib}(S^2\vee X\hookrightarrow  S^2\times X)\simeq S^1\wedge \Omega S^2\wedge \Omega X$.
Thus, if $X>S^1\wedge A$, for a connected $A$,  then:
\[
\text{Fib}(S^2\vee X\hookrightarrow  S^2\times X)>S^1\wedge S^1\wedge \Omega X>S^1\wedge S^1\wedge  A=S^2\wedge A
\]
We can then use Theorem~\ref{thm basicmoves}.(12) to infer that
 $\text{Fib}(h_X)>S^2\wedge A$. 
These two  properties of the Hurewicz map  play an  important role in this paper and therefore we are going to give them a name.
Recall that a co-augmented functor $ {\mathcal K}\colon\text{Spaces}_{\ast}\to \text{Spaces}_{\ast}$ is equipped with a natural
transformation $\mu_X\colon X\to {\mathcal K}(X)$ between $\text{id}\colon\text{Spaces}_{\ast}\to \text{Spaces}_{\ast}$ and ${\mathcal K}$.

\begin{Def}\label{cellular map}
A co-augmented functor   $\mu_X\colon X\to {\mathcal K}(X)$ is  called {\bf cellular} if:
\begin{itemize}
\item[(a)] $\text{Cof}(\eta_{X}\colon X\to \kk(X)) > S^1\wedge X$ for any connected   $X$;
\item[(b)]  $\text{Fib}(\eta_{X}\colon X\to \kk(X))>S^2\wedge A$ for any $X>S^1\wedge A$ and connected $A$.
\end{itemize}
\end{Def}
The following is a consequence  of Proposition~\ref{prop basicmoves3}.(9).

\begin{cor}\label{col:skeletal}
If $\mathcal K$ is a co-augmented  cellular functor, then $ {\mathcal K}(X)$ is $X$-cellular  for any $X$.
\hfill{\qed}
\end{cor}

\begin{example}\label{ex:cellularfunctor}
The natural unit map $X\to \Omega^n(S^n\wedge X)$ given by the  ``loop-suspension" adjunction defines a cellular functor, and so does
the map $X\to QX = \Omega^\infty\Sigma^\infty X$.  The co-augmentation $X\to SP^nX = X^n/\Sigma_n$ to the $n$-symmetric space
satisfies  the requirement (a) of~\ref{cellular map}~\cite[Section~4.A]{MR98f:55010}, and one might ask  if this  map is  also cellular.
We believe it is, although we do not have a  clear argument at the time of writing this paper. Our key example of a co-augmented cellular
functor is given by the Hurewicz map $h_X\colon X \to {\mathbb Z} X$.
\end{example}

How can we construct new co-augmented cellular functors out of old ones? For that we are going to use
 the following proposition.
Let $f\colon X\to Y$ be a map.  Take its homotopy cofiber $\alpha\colon Y\hookrightarrow \text{Cof}(f)$ and the  homotopy fiber
$\text{Fib}(\alpha)\twoheadrightarrow Y$ of $\alpha$. These maps fit into the following commutative diagram:
\[\xymatrix@R=8pt@C=18pt{
&& \text{Fib}(\alpha)\ar@{->>}[dr]\ddto|\hole \\
&X\ar@{^(->}[dd]\rrto|(.3){f}\urto^-{\overline{f}} & &Y\ar@{^(->}[dd]^{\alpha}\\
& & CX'\ar@{->>}[dr]\\
\Delta[0]&CX\ar@{->>}[l]_-{\sim}\rrto\ar@{^(->}[ur]^-{\sim} & & \text{Cof}(f)
}\]
where the front square is a push-out, the right back square is a pull-back, and the indicated maps are weak equivalence, fibrations and cofibrations.
The map $\overline{f}\colon X\to \text{Fib}(\alpha)$ is called the comparison map.

\begin{prop}\label{prop basicmoves}
Let $f: X \to Y$ be a map, $\overline{f}$ as above, and  $A$ be a connected space.
\begin{enumerate}
\setcounter{enumi}{\value{tempcounter}}
\item If $X$ is connected and $\text{\rm Cof}(f)$ is simply connected, then $\text{\rm Cof}(\overline{f})\gg S^1\wedge X$.
\item If $\text{\rm Fib}(f)$ is simply connected and $X>S^1\wedge A$, then $\text{\rm Fib}(\overline{f})>S^2\wedge A$.
\end{enumerate}
\end{prop}

\begin{proof}
We prove (15) first. By the assumption $\text{Cof}(f)$ is $1$-connected. Thus its loop space, which is the homotopy
fiber of $CX\to\text{Cof}(f)$, is connected. We can therefore apply Theorem~\ref{thm basicmoves}.(13) to the following homotopy push-out square
\[\xymatrix{
X\ar@{^(->}[d]\rto^-f & Y\dto\ar@{^(->}[d]^{\alpha}\\
CX\rto & \text{Cof}(f)
}\]
to get $\text{Cof}(\overline{f})\gg S^1\wedge X\wedge \Omega^2\text{Cof}(f)\gg S^1\wedge X$.

To prove (16), assume $X>S^1\wedge A$.  This  implies  $X$ is $1$-connected.
According to Proposition~\ref{prop basicmoves2}.(7) $\text{Fib}(\alpha)\gg X$ and thus $\text{Fib}(\alpha)$
is also $1$-connected. The hypothesis of Theorem~\ref{thm basicmoves}.(14) is thus satisfied and we get $\text{\rm Fib}(\overline{f})>S^1\wedge \Omega X\wedge \Omega\text{Fib}(f)$.
Since $\text{Fib}(f)$ is $1$-connected,  $ \Omega\text{Fib}(f)$ is connected and we can conclude
$
\text{\rm Fib}(\overline{f})>S^1\wedge \Omega X\wedge \Omega\text{Fib}(f)>S^2\wedge \Omega X>S^2\wedge A
$.
\end{proof}

Given a co-augmented functor $\calk$, the above construction allows us to construct a new functor
$\overline{\calk}\colon\text{Spaces}_{\ast}\to \text{Spaces}_{\ast}$. For any space $X$, the co-augmentation
$\mu_{X}\colon X\to \calk(X)$ induces namely a map $ \overline{\mu_X}\colon X\to \overline{\calk}(X)$.

Here is the key statement of this section:
\begin{prop}\label{prop ovkB-cellular}
Let $\mu_{X}\colon X\to \calk(X)$ be a co-augmented functor. Assume:
\begin{itemize}
\item If $X$ is connected, then $\text{\rm Cof}(\mu_{X})$
is simply connected.
\item If $X$ is simply connected, then so is $\calk(X)$ and $\pi_2(\mu_X)$ is an epimorphism.
\end{itemize}
 Then
 $\overline{\calk}$ is a cellular  co-augmented functor and  $\overline{\calk}(X)$ is $X$-cellular for any $X$.
\end{prop}

\begin{proof}
Requirement (a) of Definition~\ref{cellular map} follows from Proposition~\ref{prop basicmoves}.(15)
and requirement (b)  from Proposition~\ref{prop basicmoves}.(16). The fact that $\overline{\calk} (X)$ is
$X$-cellular is a direct consequence of~\ref{col:skeletal}.
\end{proof}

For example the assumptions of Proposition~\ref{prop ovkB-cellular} are satisfied if $\calk$ is   cellular:
\begin{cor}\label{cor overlinecellular}
If $\calk$ is a co-augmented cellular functor, then  so is $\overline{\calk}$. \hfill{\qed}
\end{cor}

\section{The modified Bousfield-Kan tower}
\label{sec injtwist}
The aim of this section is to show that the co-augmented functors  $\mu_{k,X}\colon X\to z_kX$ in a  modified version of the integral
Bousfield-Kan  completion tower, as defined by the second author in~\cite{MR2002e:55012}, are cellular (see Definition~\ref{cellular map}).
The modified tower  was built originally as an elementary construction that models the
pro-homology type of any space by a tower of much simpler spaces called polyGEMs.

\begin{Def} \label{polyGEM}
A  \emph{$1$-{\bf polyGEM}} is defined to be a GEM, i.e., a space weakly equivalent to a product of abelian Eilenberg-MacLane spaces.
For $n\geq 2$, an $n$-{\bf polyGEM} is a space which is weakly equivalent to a retract of the homotopy fiber of
a map from an $(n-1)$-polyGEM to a GEM. A space is a {\bf polyGEM} if it is an $n$-polyGEM for some integer~$n$.
\end{Def}

PolyGEMs are examples of nilpotent spaces. They are in a  sense universal such examples:
\begin{prop}\label{cor nilpotnent}
A connected space $X$ is nilpotent if and only if, for any $n\geq 1$, the $n$-th Postnikov section $P_nX$ is a {\rm polyGEM}.
\end{prop}

\begin{proof}
If $X$ is nilpotent its Postnikov tower admits a refinement by principal fibrations whose fibers are Eilenberg-MacLane spaces. Conversely, if $P_n X$ is a polyGEM, it is nilpotent. Hence its fundamental group is nilpotent and acts nilpotently on all homotopy
groups. This is true for all integers $n$, so $X$ itself is nilpotent.
\end{proof}

How can we detect that a space is a polyGEM? This can be done using the
modified version of the integral Bousfield-Kan homology completion tower:
\[\xymatrix{
& & X\ar@/^10pt/[drr]^{\mu_{0,X}}\drto|{\mu_{1,X}}\ar@/_10pt/[dl]_{\mu_{k,X}}\\
\cdots\ar@{->>}[r]^(.4){q_{k+1}} & z_k X\ar@{->>}[r]^-{q_k} &\cdots\ar@{->>}[r]^(.4){q_2} & z_1 X\ar@{->>}[r]^(.4){q_1} & z_0X
}\]
Recall from~\cite{MR2002e:55012} the inductive construction of the tower.
For $k=0$, $z_0 X= {\mathbb Z} X$  and
$\mu_{0,X}\colon X\to z_0 X$ is the Hurewicz map $h_X\colon X\to {\mathbb Z}X$.  For $k\geq 0$, the space   $z_{k+1}X$ is the homotopy fiber of the composition:
\[z_{k}X\xrightarrow{\alpha} \text{Cof}(\mu_{k,X})\xrightarrow{h_{\text{Cof}(\mu_{k,X})}} {\mathbb Z}\text{Cof}(\mu_{k,X})\] and the
 map $\mu_{k+1,X}\colon X\to z_{k+1} X$ fits into the following commutative diagram:
 \begin{equation}\label{diagram def}\xymatrix{
 X\ar@/^18pt/[drr]^{\mu_{k,X}}\drto^-{\overline{\mu_{k,X}}}\ar@/_18pt/[ddr]_{\mu_{k+1,X}} \\
&\overline{z_k}X\ar@{->>}[r] \dto &  z_{k} X\ar@{^(->}[r]^-{\alpha } \ar@{=}[d] &\text{Cof}(\mu_{k,X})\ar@{^(->}[d]^{h_{\text{Cof}(\mu_{k,X})}}\\
& z_{k+1} X\ar@{->>}[r]^{q_{k+1}} & z_{k} X\ar@{^(->}[r]^-{h\alpha} &  {\mathbb Z}\text{Cof}(\mu_{k,X})\\
  }\tag{$\ast$}\end{equation}

Observe that $z_0 X$ is a GEM by definition, $z_1X$ is a $2$-polyGEM as it is the homotopy fiber of a map between GEMs, and,
more generally $z_k X$ is a $(k+1)$-polyGEM for any $k\geq 0$, as it is by induction the homotopy fibre of a map from a $k$-polyGEM
to a GEM. Moreover, this new tower mimics the behavior of the classical Bousfield-Kan tower in the following sense:

\begin{prop} \label{tower}
For any space $X$ the maps  $\mu_{k,X}\colon X\to z_k X $
induce  a pro-homology and cohomology isomorphism between the constant tower $X$ and
the modified integral Bousfield-Kan tower $(\cdots \twoheadrightarrow z_1X\twoheadrightarrow z_0X)$.
Moreover, if $X$ is a polyGEM, then this map of towers induces a pro-isomorphism on pro-homotopy groups.
\end{prop}

\begin{proof}
The pro-homology isomorphism (and therefore also pro-cohomology isomorphism) holds by \cite[Theorem~2.2]{MR2002e:55012}.
As a corollary, \cite[Proposition~2.13]{MR2002e:55012}, the map of towers induces a pro-isomorphism on pro-homotopy groups for any
polyGEM, i.e. the kernel and cokernel are pro-isomorphic to zero.
\end{proof}

We can now state our  detection principle.

\begin{prop} \label{polyretract}
A space  $W$ is a  polyGEM if and only if it is a retract of $z_nW$ for  some~$n$.
\end{prop}

\begin{proof}
Since $z_nW$ is a polyGEM, by definition so is any of its retracts. That proves one implication.

Assume now that  $W$ is a  polyGEM.
The maps $\mu_{k,W}\colon W\to z_k W $  induce  a pro-cohomology equivalence as we just have seen in Proposition~\ref{tower}.
Thus, for a fibrant GEM $P$, the maps of mapping spaces $\textrm{map}_{\ast}(\mu_{k,W},P)\colon \textrm{map}_{\ast}(z_k W,P) \to
\textrm{map}(W,P)$  induces an ind-homotopy equivalence.
By induction, the same holds for any fibrant polyGEM~$P$. To obtain the desired retraction, we can now use this ind-homotopy equivalence
when $P$ is a fibrant replacement of~$W$.
\end{proof}

We conclude this section with its main result.

\begin{prop}\label{prop twistB-cellular}
For all $k\geq 0 $,  the co-augmented functor  $z_k$ is  cellular. In particular $z_kX$ is $X$-cellular for any $X$.
\end{prop}

\begin{proof}
The proof is by induction on $k$.
The cellularity of  $z_0=\mathbb{Z}$ was already discussed in Section~\ref{sec Bousfield}.
Assume that $k>0$. Let us denote  by $\beta\colon \overline{z}_k X\to
z_{k+1} X$ the left vertical map in the diagram (\ref{diagram def}). This map fits into a commutative triangle:
\[\xymatrix{
X\drto^{\overline{\mu_{k,X}}}\ar@/_18pt/[ddr]_{\mu_{k+1,X}} \\
& \overline{z_k}X\dto^{\beta}\\
& z_{k+1} X
}\]
which exhibits $\mu_{k+1,X}$ as a composition of two maps,  yielding both a cofibration and a fibration sequence:
\[
\text{Cof}(\overline{\mu_{k,X}})\to \text{Cof}(\mu_{k+1,X})\to  \text{Cof}(\beta)\  \  \  \  \  \  \ \
\text{Fib}(\overline{\mu_{k,X}})\to \text{Fib}(\mu_{k+1,X})\to  \text{Fib}(\beta)
\]

Assume now that $X$ is connected. Since the functor $z_k$ is  cellular (by induction),  $\text{Cof}(\mu_{k,X})>S^1\wedge  X$ and
 we get:
\[
\Gamma \text{Cof}(\mu_{k,X}):=  \text{Fib}\left( \text{Cof}(\mu_{k,X})\xrightarrow{h_{\text{Cof}(\mu_{k,X})}}\mathbb{Z}\text{Cof}(\mu_{k,X})\right)>S^2\wedge  X
\]
From the  cellularity of $\mathbb Z$ (part (b) of  Definition~\ref{cellular map}). It follows that:
\[
\text{Fib}\left( \beta\colon \overline{z}_k X\to  z_{k+1} X \right)\simeq \Omega\Gamma \text{Cof}(\mu_{k,X})>\Omega
(S^2\wedge  X)>S^1\wedge X
\]
Consequently $\text{Cof}(\beta)>S^1\wedge \text{Fib}(\beta)>S^2\wedge X$ (see Proposition~\ref{prop basicmoves3}.(8)).
As $\overline{z_k}$ is  cellular (see Corollary~\ref{cor overlinecellular}), we also have $\text{Cof}(\overline{\mu_{k,X}})>S^1\wedge X$.
These last two inequalities imply $\text{Cof}(\mu_{k+1,X})>S^1\wedge X$ which is requirement (a) of Definition~\ref{cellular map}.

 Assume now that $X>S^1\wedge A$ for a connected space $A$. Since  $\overline{z_k}$ is  cellular,
 $\text{Fib}(\overline{\mu_{k,X}})>S^2\wedge A$. We have already seen that  $\text{Fib}(\beta)>S^1\wedge X$, and hence  $\text{Fib}(\beta)>S^2\wedge A$.
 These  inequalities imply  $\text{Fib}(\mu_{k+1,X})>S^2\wedge A$, which is requirement (b) of Definition~\ref{cellular map}.
 This concludes the induction step and the proof of the proposition.
\end{proof}

\section{Cellularity of Postnikov sections}
 Recall that the Postnikov sections and highly connected covers are the most basic occurrences of nullifications and cellular covers:   
 $P_{S^{n+1}}X$ is the $n$-th Postnikov section $P_nX$ and $\text{\rm cell}_{S^{n+1}}X$ is the $n$-connected cover.
We are ready now to prove the main theorem of this article: a space
builds any of its nilpotent Postnikov sections, even if the space itself is not nilpotent. This is in contrast with
highly connected covers  as the following examples illustrates.

\begin{example}\label{ex connectedcover}
In general, an $n$-connected cover $\text{cell}_{S^{n+1}}X$ -- even of  a  nilpotent space $X$ --  is  not  $X$-cellular. Consider
for example  $K({\mathbb Z},2) \vee K({\mathbb Z},2)$ whose cellular class is that of $K({\mathbb Z}, 2)$.
Its $2$-connected cover is $S^3$ and the $3$-sphere is not   $K({\mathbb Z},2)$-cellular, in fact not  $K({\mathbb Z},2)$-acyclic.
To see this  let  $M({\mathbb Z}/2,3)$ be the double  suspension of ${\mathbb R}P^2$. This is a  finite complex
so, by the Sullivan conjecture, \cite{MR750716}:
\[
\text{map}_{\ast}(K({\mathbb Q}/{\mathbb Z},1),M({\mathbb Z}/2,3))\simeq \ast
\]
The space  $K({\mathbb Q}/{\mathbb Z},1)$ is the homotopy fiber of the map $K({\mathbb Z},2)\to K({\mathbb Q},2)$
induced by the inclusion ${\mathbb Z}\subset{\mathbb Q}$. By the above consequence of the Sullivan conjecture we then have:
\[
\text{map}_{\ast}(K({\mathbb Z},2), M({\mathbb Z}/2,3))\simeq \text{map}_{\ast}(K({\mathbb Q},2), M({\mathbb Z}/2,3))
\]
However, since $K({\mathbb Q},2)$ is a rational space, $\text{map}_{\ast}(K({\mathbb Q},2), M({\mathbb Z}/2,3))$ is contractible
and consequently so is
$\text{map}_{\ast}(K({\mathbb Z},2), M({\mathbb Z}/2,3))$.
This shows that $M({\mathbb Z}/2,3)$ can not be $K({\mathbb Z},2)$-acyclic.
As $M({\mathbb Z}/2,3)>S^3$, $S^3$ can not be $K({\mathbb Z},2)$-acyclic
either.
\end{example}

In the proof of our main result  the following cellular and acyclic properties of
polyGEMs play key roles. In  Example~\ref{ex connectedcover} we have seen that in general
 the $n$-connected cover $\text{\rm cell}_{S^{n+1}}X$ of a nilpotent space $X$ can fail to be  $X$-acyclic. This however can not happen when $X$ is a  polyGEM:

\begin{prop}\label{cor main}
Let $W$ be a polyGEM. Then,  for any $n\geq 0$:
\begin{enumerate}
\item$\text{\rm cell}_{S^n}W>W$;
\item  $K(\pi_nW,n)>W$;
\item $P_{n}W\gg W$ and in particular $K(\pi_1W,1)\gg W$;
\item $\overline{\calc(W)}=\overline{\calc({\z}W)}=\overline{\calc(\prod_{k\geq 0}K(\pi_k W,k))}$
\end{enumerate}
\end{prop}

\begin{lemma}\label{lem acycconcofib}
Let $X$ be a space. For any $k\geq 0$ and   $n\geq 0$,    $\text{\rm cell}_{S^{n}} {z}_{k}X>X$.
\end{lemma}

\begin{proof}
If $X$ is not connected, then the lemma is clear as all spaces are $X$-acyclic. Assume $X$ is connected.
The proof is by  induction  on $k$.
For $k=0$, the space $ {z}_{0}X={\z}X$ is a GEM. Thus for any $n\geq 0$,
$\text{cell}_{S^{n}}{\z}X$ is a retract of ${\z}X$ and  since ${\z}X$ is $X$-cellular,
then so is $\text{cell}_{S^{n}}{\z}X$.

Assume $k>0$. Since $X$ is connected, then so is $z_kX$ and hence $\text{cell}_{S^{0}}z_kX$ and
$\text{cell}_{S^{1}}z_kX$  are weakly equivalent to $z_kX$, which is $X$-acyclic (even cellular) by
Proposition~\ref{prop twistB-cellular}.
Assume $n\geq 2$ and form the following commutative diagram where the horizontal sequences are fibration sequences
and the left and right vertical maps are cellular covers (see~\cite[Theorem E2]{MR98f:55010}):
\begin{equation}\label{diagramcovers}
\xymatrix{
\text{cell}_{S^n} {z}_{k}X\dto\rto & E\dto\rto &  \text{cell}_{S^{n+1}}{\z}\text{Cof}(\mu_{k-1,X})\dto\\
 {z}_{k}X\rto^-{q_k} &  {z}_{k-1}X\rto &{\z}\text{Cof}(\mu_{k-1,X})
}\tag{$\ast\ast$}
\end{equation}
Explicitly, the space $E$ is the $(n-1)$-connected cover of the homotopy pull-back  of the right hand side pull-back
diagram.
Since ${z}_{k-1}$ is a  cellular functor (see Proposition~\ref{prop twistB-cellular}),  $\text{Cof}(\mu_{k-1,X})>S^1\wedge X$.
This, together with  the case  $k=0$, gives
$\text{cell}_{S^{n+1}}{\z}\text{Cof}(\mu_{k-1,X})>\text{Cof}(\mu_{k-1,X})>
S^1\wedge X$ or equivalently by Proposition~\ref{prop basicmoves1}.(2)
$\Omega \text{cell}_{S^{n+1}}{\z}\text{Cof}(\mu_{k-1,X})> X$.
To show  $\text{cell}_{S^n} {z}_{k}X>X$ it is therefore enough to prove
 $E>X$ (see Proposition~\ref{prop basicmoves3}.(8)).

The space $E$ is $(n-1)$-connected and the map $E \to z_{k-1} X$ induces an isomorphism
on homotopy group $\pi_i$ for $i \geq n+1$. We have thus a fibration sequence:
\[
\text{cell}_{S^{n+1}} {z}_{k-1}X \to E\to K(\pi_n E,n)
\]
By the inductive assumption  $\text{cell}_{S^{n+1}} {z}_{k-1}X>X$. The inequality
$E>X$ will then follow once we show  $K(\pi_n E,n)>X$.

Let $G=\pi_{n} {z}_{k-1}X$, $H=\pi_{n}{\z}\text{Cof}(\mu_{k-1,X})$ and $f\colon G\to H$ be the group
homomorphism induced on $\pi_n$ by the map $ {z}_{k-1}(X)\to {\z}\text{Cof}(\mu_{k-1,X})$.
By the inductive assumption  $\text{cell}_{S^{n}} {z}_{k-1}X>X$ and
 $\text{cell}_{S^{n+1}} {z}_{k-1}X>X$. These spaces fit into the following fibration sequence:
\[
\text{cell}_{S^{n+1}} {z}_{k-1}X\to \text{cell}_{S^{n}} {z}_{k-1}X\to K(G,n) \, ,
\]
It follows that $K(G,n)>X$. As $K(H,n)$ is a retract of  ${\z}\text{Cof}(\mu_{k-1,X})$,
we also have  $K(H,n)> \text{Cof}(\mu_{k-1,X})>S^1\wedge X$.
These inequalities imply:
\[
K(\text{Ker}(f),n)\times K(\text{Coker}(f),n-1) \simeq \text{Fib}\left(K(f, n)\colon K(G,n)\to K(H,n)\right)> X.
\]
Hence, as a retract of $\text{Fib}\left(K(f, n)\right)$, the space $K(\text{Ker}(f),n)$ is also $X$-acyclic.
The long exact sequences in homotopy for the fibrations in the above diagram (\ref{diagramcovers}) allow us to identify
$\pi_n E$ with $\text{Ker}\left(f\colon G\to H \right)$. We conclude that $K(\pi_n E,n)>X$.
\end{proof}

\begin{proof}[Proof of Proposition~\ref{cor main}]
If $W$ is not connected, then all the four statements are clear. Assume then that $W$ is connected.

\medskip

\noindent
(1):\quad  Since $W$ is a polyGEM, Proposition~\ref{polyretract} implies that it is a retract of $ {z}_{k}W$ for some~$k$.
By functoriality, $\text{cell}_{S^n}W$ is then a retract of $\text{cell}_{S^n} {z}_{k}W$
and we conclude by Lemma~\ref{lem acycconcofib} that $\text{cell}_{S^n}W>W$.
\medskip

\noindent
(2):\quad  This is a consequence of (1) and the fact that we have a fibration sequence:
\[
\text{cell}_{S^{n+1}}W\to \text{cell}_{S^n}W\to K(\pi_nW,n)
\]

\noindent
(3):\quad  For $n = 0$ the result is immediate as
$P_{0}W$ is a retract of $W$. Let $n\geq 1$. In this case
the statement follows from (1) and Proposition~\ref{prop basicmoves3}.(10) applied to the fibration sequence $\text{cell}_{S^{n+1}}W\to W\to P_{n}W$.
\medskip

\noindent
(4):\quad   We start by showing by induction that $ {z}_{k}W>{\z}W$.
For $k=0$ there is nothing to prove. Assume that $k\geq1$.
Recall that $ {z}_{k}W$ fits into a fibration sequence:
\[
\Omega {\z}\text{Cof}(\mu_{k-1,W})\to {z}_{k}W\to  {z}_{k-1}W
\]

As $ {z}_{k-1}$ is  cellular (see Proposition~\ref{prop twistB-cellular}), $\text{Cof}(\mu_{k-1,W}) >S^1\wedge W$
and hence ${\z}\text{Cof}(\mu_{k-1,W}) >{\z}(S^1\wedge W)$ which implies  $\Omega {\z}\text{Cof}(\mu_{k-1,W})>{\z}W$. By induction we also have
 ${z}_{k-1}W>{\z}W$. We can then conclude ${z}_{k}W>{\z}W$.
 Note that in the above argument we did not use the assumption that $W$ is a polyGEM and hence this acyclic inequality is true for an arbitrary space.
 However in the case $W$ is a polyGEM, there is an integer $k$ for which $W$ is a retract of ${z}_{k}W$.
 For such a $k$ we have  then the following relations
$W\gg  {z}_{k}W>{\z}W\gg W$ which proves the equality $\overline{\calc(W)}=\overline{\calc({\z}W)}$.

The inequality $ \prod K(\pi_k W,k))> W$ follows from  statement (2).
To conclude that the  acyclic classes
$\overline{\calc(W)}$ and $\overline{\calc(\prod K(\pi_k W,k))}$ coincide what remains is the proof of  the relation ${\z}W>\prod K(\pi_kW,k)$.
For any $n \geq 0$, the inequality $K(H_n(W),n)\gg {\z}(P_{n}W)$ holds
since $W$ and $P_{n}W$ have isomorphic  $n$-th integral homology groups by the Whitehead Theorem.
Therefore:
\[
K(H_n(W),n)\gg {\z}(P_{n}W) \gg P_{n}W>\prod_{k\geq 0} K(\pi_kW,k)
\]
This implies  that ${\z}W>\prod_{ k\geq 0} K(\pi_kW,k)$.
\end{proof}

\begin{thm}\label{thm main}
Let $X$ be a space. If $P_{n}X$
 is nilpotent, then $P_{n}X\gg X$.
\end{thm}

\begin{proof}
If $X$ is not connected, then the conclusion is clear.  Assume thus  that $X$ is connected
and  $P_{n}X$ is nilpotent, which by Proposition~\ref{cor nilpotnent} means that   $P_{n}X$ is a polyGEM. Proposition \ref{polyretract} implies the existence of an integer $k$ and of a map $r\colon  {z}_{k}P_n X\to W$
into a fibrant space $W$ such that the composition with the augmentation ${\mu}_{k,P_n X} \colon P_n X\to  {z}_{k}P_n X$ is a weak equivalence.
We can then form the following  diagram which is commutative if we remove the dotted arrows
and in which the symbol $p$ denotes the relevant   functorial $n$-th Postnikov section maps:
\[\xymatrix@C=30pt{
X\rrto^{p}\dto_{{\eta_X}} & & P_nX\dto^{\eta_{P_n X}}  \ar@{.>}@/_15pt/[ddll] \ar@/^13pt/[rd]^{\sim}\\
  {z}_{k}X\rrto|(.28)\hole^(.65){  {z}_{k}p}\dto_{p_{  {z}_{k}X}} &  &  {z}_{k}P_nX
 \dto^{p}\rto^-{r} & W\\
P_n  {z}_{k}X\rrto^{P_n {z}_{k}p} & &P_n {z}_{k}P_nX \ar@{.>}@/_13pt/[ru]
}\]
The maps represented by the dotted arrows, which make the entire diagram homotopy commutative, exist by the universal property of the Postnikov sections. We can therefore conclude that $P_nX$ is a retract of
$P_n  {z}_{k}X$. We thus have the following cellular inequalities:
\[
P_nX\gg P_n  {z}_{k}X \gg  {z}_{k}X \gg X\]
where the second cellular inequality  follows from Proposition~\ref{cor main}.(3) together with the fact that $z_k X$ is
a polyGEM and the last one is given in Proposition~\ref{prop twistB-cellular}.
\end{proof}

The following is a particular case of Theorem~\ref{thm main}, for $n=1$. Hints that this result could hold motivated
the present work.

\begin{cor}\label{cor nilpotnentfundgr}
If  $\pi_1(X)$ is  nilpotent, then $K(\pi_1X,1)\gg X$. \hfill{\qed}
\end{cor}

We can also use Theorem~\ref{thm main} to get a Serre class-type statement that describes a global relation between
the integral homology groups and the homotopy groups of a nilpotent space. No spectral sequence is needed in
our proof, even though it seems that one could also obtain the mutual acyclicity of the homotopy and homology groups
by a spectral sequence argument.

\begin{cor}
If $X$ is  nilpotent, then $\overline{\calc\left({\z}(X)\right)}=\overline{\calc(\prod_{k\geq 0}K(\pi_k X,k))}$.
\end{cor}

\begin{proof}
If $X$ is not connected, the statement is clear. Assume thus $X$ is connected.
Even without the nilpotency assumption on $X$, for any $n\geq 0$, we have:
\[
K(H_n(X),n)=K(H_n(P_{n}X),n)\gg {\z}(P_{n}X)\gg P_{n}X> \prod_{k\geq 0} K(\pi_kX ,k)
\]
Consequently ${\z}(X)>\prod_{k \geq 0} K(\pi_kX ,k)$. For the opposite inequality, we  need the assumption  $X$ is nilpotent,
which according to Proposition~\ref{cor nilpotnent} is equivalent to
 $P_nX$ being a polyGEM for any $n\geq 0$.  We can then use Theorem \ref{thm main} and Corollary \ref{cor main}(2) to obtain
$
K(\pi_nX,n)>P_{n+1}X\gg X
$ which, for $n>1$,  implies  $K(\pi_nX,n)\gg {\z}(K(\pi_nX,n))>{\z}(X)$.
For $n=1$, since $\pi_1X$ is nilpotent, we also have $K(\pi_1X,1)>{\z}(X)$.
This shows  $\prod_{k \geq 0} K(\pi_kX ,k)>{\z}(X)$.
\end{proof}

\section{Applications}
In this section we state various consequences of Proposition~\ref{prop twistB-cellular} and Theorem~\ref{thm main}.
We start with the preservation of polyGEMS by general cellularizations:

\begin{thm}\label{thm preservationpolygem}
If $X$ is a polyGEM, then so is $\text{\rm cell}_AX$ for any space $A$.
\end{thm}

\begin{proof}
Assume   $X$ is a polyGEM. It is thus a retract of $z_n X$ (see~\ref{polyretract}), for some $n$,  and hence we can form
the following  diagram with the indicated maps being weak equivalences and which, if the dotted arrow is removed,  is
commutative by the functoriality of the constructions:
\[
\xymatrix{
\text{cell}_AX\rrto^(.55){\mu_{n,\text{cell}_AX}}\drto|{c_{A,X}} \ar@/^25pt/[rrrr]|{\text{cell}_A\mu_{n,X}}\ar@/^40pt/[rrrrrr]|{\simeq}
&& z_n\text{cell}_AX\ar@{.>}[rr]\drto|{z_n(c_{A,X})} & & \text{cell}_Az_nX\rrto^-{\text{cell}_Ar}\dlto|{c_{A,z_nX}} && \text{cell}_AY\dlto|{c_{A,Y}}\\
&X\rrto^{\mu_{n,X}} \ar@/_15pt/[rrrr]|{\simeq} & &  z_nX\rrto^{r} & & Y
}\]
Since $z_n$ is a  cellular functor (see Proposition~\ref{prop twistB-cellular}), $z_n\text{cell}_A X \gg \text{cell}_A X \gg A$.
We can then use the  universal property of the $A$-cellular cover $c_{A,z_nX}$ to get the  existence of  the  dotted arrow
that makes the  entire diagram homotopy commutative.
Commutativity of this diagram shows that $\text{cell}_A X$ is a retract of $z_n \text{cell}_A X$ and
hence by Proposition~\ref{polyretract}   $\text{cell}_A X$ is a polyGEM.
\end{proof}

An analogous  result  to Theorem~\ref{thm preservationpolygem}  holds also for finite nilpotent  Postnikov stages.
 \begin{cor}\label{cor npostpreserv}
 If $X$ is a  nilpotent $n$-Postnikov stage, then so is $\text{\rm cell}_A X$ for any $A$.
\end{cor}

\begin{proof}
Assume  $p_{n,X}\colon X\to P_nX$ is a weak equivalence and $P_nX$ is nilpotent.
Consider the following diagram where the indicated maps are weak equivalences and,
if the dotted arrow is removed, it   is commutative by functoriality of the constructions:
\[\xymatrix{
\text{cell}_AX\dto_{c_{A,X}}\rrto_-{p_{n,\text{cell}_AX}} \ar@/^22pt/[rrrr]_{\simeq}^{\text{cell}_Ap_{n,X}} &&
P_n\text{cell}_AX\drto|{P_nc_{A,X}}\ar@{.>}[rr] & & \text{cell}_AP_nX\dlto|{c_{A,P_nX}}   \\
X\ar[rrr]_{\simeq}^{p_{n,X}} & && P_nX
}\]
Since $X$ is a polyGEM (see~\ref{cor nilpotnent}), then so is  $\text{cell}_A X$
by Theorem~\ref{thm preservationpolygem}. We can then use Proposition~\ref{cor main}.(3) to conclude that
$P_n \text{cell}_A X\gg \text{cell}_A X\gg A$. The universal property of the cellular cover
$c_{A,P_nX}$ gives  the existence of the dotted arrow making the above diagram homotopy commutative.
This implies that the map $p_{n,\text{cell}_AX}\colon \text{cell}_AX\to P_n\text{cell}_AX$ induces a monomorphism on all
homotopy groups. As it also induces an epimorphism, it is a weak equivalence and
$\text{cell}_AX$ is an $n$-Postnikov stage.
\end{proof}

Theorem~\ref{thm preservationpolygem} can   be used to give a  description of all cellular covers of the classifying space of a
nilpotent group in terms of the group theoretical covers~\cite{MR2269828},\cite{MR2483285},~\cite{MR2995665}:

\begin{cor}\label{cor orbitBnilgr}
Assume that   $G$ is a nilpotent group and $A$ a connected space. Then
$\text{\rm cell}_AK(G,1)\simeq
K(\text{\rm cell}_{\pi_1A}G,1)$ where $\text{\rm cell}_{\pi_1A}G$ is the group theoretical $\pi_1A$-cellularization of $G$
(see~\cite{MR2269828,MR2483285,MR2995665}).
\end{cor}

\begin{proof}
Corollary~\ref{cor npostpreserv} implies that $\text{\rm cell}_AK(G,1)\simeq K(H,1)$  where $H$ is a nilpotent group. Furthermore we claim
that since $K(H,1)$ is $A$-cellular, the group $H$ is $\pi_1A$-cellular. To see this note that, by the Seifert-van Kampen Theorem,
the collection of all connected spaces with $\pi_1A$-cellular  fundamental group is a cellular class. Since it contains $A$, it has to include
the smallest cellular  collection ${\mathcal C}(A)$, and in particular  it contains  $K(H,1)$.

By the universal property of the cellularization and Theorem~\ref{thm excellul}.(A), the following map is  a weak   equivalence:
\[
\text{map}_{\ast}(A, c_{A,K(G,1)})\colon \text{map}_{\ast}(A, K(H,1))\to \text{map}_{\ast}(A, K(G,1))
\]
Thus, on $\pi_1$, we get that the homomorphism  $\pi_1c_{A,K(G,1)}\colon H\to G$ induces a bijection:
\[
\text{Hom}(\pi_1A, \pi_1c_{A,K(G,1)})\colon \text{Hom}(\pi_1A, H)\cong \text{Hom}(\pi_1A, G)
\]
This homomorphism  $\pi_1c_{A,K(G,1)}\colon H\to G$ is therefore the $\pi_1A$-cellularization and $H$ is isomorphic to
$\text{\rm cell}_{\pi_1A}G$.
\end{proof}

By  \cite[Theorem~1.4 (2)]{MR2269828}, the group theoretical cellularizations of a finite nilpotent group $N$ are always subgroups of $N$.
It is sometimes possible to compute all possible such cellularizations.

\begin{example}
\label{ex dihedral}
Let $D_{2^n}$ denote the dihedral group of order $2^n$ for $n \geq 2$. This is the group of symmetries of a regular polygon with $2^{n-1}$
sides and it is nilpotent of class $n-1$. The third author showed in \cite[Proposition 5.1]{Ramon} that $K(D_{2^n},1)$ is $\mathbb Z/2$-cellular. We have
hence only two possible cellularizations, $\text{\rm cell}_A K(D_{2^n},1)$ can be contractible or $\text{\rm cell}_A K(D_{2^n},1)\simeq K(D_{2^n}, 1)$.
The latter is obtained for example for $A = K(\mathbb Z/2, 1)$. For $n\leq 3$, we were able to perform these computations by hand, but already
for $n=4$ we do not know of a direct calculation of all cellularizations of $K(D_{16},1)$ without showing first that they must be $K(G, 1)$'s.
\end{example}

We can  also  strengthen Corollary~\ref{cor nilpotnentfundgr}.  Instead of assuming
that the fundamental group of $X$ is nilpotent we make it nilpotent by taking the quotient by some stage of the lower central
series.

\begin{cor}
\label{cor quotient}
For any $n\geq 1$,  $K(\pi_1X/ \Gamma_n \pi_1 X ,1)\gg X$.
\end{cor}

\begin{proof}
Let $G=\pi_1X$.
If $X$ is not connected, the corollary is clear. Assume $X$ is connected. In this case
according to Corollary~\ref{cor orbitBnilgr}, $\text{cell}_{X}K(G/ \Gamma_n G ,1)\simeq K(\text{cell}_{G}(G/ \Gamma_nG),1)$.
However $G/ \Gamma_nG$ is  a $G$-cellular group (see~\cite[Proposition 7.1.(3)]{MR2483285}). We can then conclude
$\text{cell}_{X}K(G/ \Gamma_n G ,1)\simeq  K(G/ \Gamma_n G ,1)$, which proves  that $K(G/ \Gamma_n G ,1)$ is $X$-cellular.
\end{proof}

The statement of the next result  does not involve cellularity, however we do not know of a proof which does not use our techniques.
This is the extension to nilpotent fundamental groups of the Bousfield Key Lemma we presented in the introduction.
We recalled the precise statement of the key lemma at the beginning of Section~\ref{sec Bousfield}. It implies for example that,  if $X$ Êis simply
connected, then  the map $\pi_n\colon\text{map}_{\ast}(X,X)\to \text{Hom}(\pi_nX,\pi_nX)$ is a weak equivalence if and only if $X$Ê
is weakly equivalent to $K(\pi_nX,n)$. If  $X$ Êis not simply connected,  then the situation is much more complicated.

\begin{example}
\label{ex symmetric}
The $K(\mathbb Z/2, 1)$-cellularization of $K(\Sigma_3, 1)$ has been computed in~\cite[Example 2.6]{MR2351607}. It is a space
$X$ whose fundamental group is the symmetric group $\Sigma_3$ and its universal cover is the homotopy fiber of
the degree $3$ map on the sphere $S^3$. In particular its homotopy groups are non trivial in infinitely many degrees.
By the universal property of the cellularization we have weak equivalences of mapping spaces
\[
\text{map}_{\ast}(X,X) \simeq \text{map}_{\ast}(X,K(\Sigma_3, 1)) \simeq \text{Hom}(\Sigma_3,\Sigma_3)
\]
The mapping space of pointed self-maps of $X$ is homotopically discrete, but $X$ is far from being a $K(G, 1)$.
This example also shows that the cellularization of non-nilpotent spaces can become quite complicated.
\end{example}

This is in contrast with what we prove for nilpotent fundamental groups. If the pointed mapping space of self-maps
$\text{map}_{\ast}(X,X)$ is homotopically discrete via the evaluation on the fundamental group, then $X$ must
by a $K(G, 1)$.

\begin{thm}\label{cor maindiscmapsp}
Let $X$ be a connected space whose fundamental group  $\pi_1X$ is nilpotent. Assume that
the map $\pi_1\colon\text{map}_{\ast}(X,X)\to\text{Hom}(\pi_1X,\pi_1X)$ is a weak equivalence.
Then $X$ is weakly equivalent to $K(\pi_1X, 1)$.
\end{thm}

\begin{proof}
The assumptions imply that $\text{map}_{\ast}(X,p_{1,X})\colon \text{map}_{\ast}(X,X)\to \text{map}_{\ast}(X, P_1X)$ is a weak equivalence.
Thus $X\simeq \text{cell}_XK(\pi_1X,1)$, which by Corollary~\ref{cor orbitBnilgr}, means that $X\simeq K(\pi_1X,1)$.
\end{proof}

Here is another way to restate this result. If the first Postnikov section $X \rightarrow K(G, 1)$ induces a
weak equivalence on pointed mapping spaces  $\text{map}_{\ast}(X,X) \simeq \text{map}_{\ast}(X,K(G, 1))$, then
$X$ is a $K(G, 1)$.
We also offer a version for higher Postnikov sections. The very same argument as in the proof of Theorem~\ref{cor maindiscmapsp}
can be used to show:

\begin{cor}
\label{thm mapping2}
Let $X$ be a connected space whose $n$-th Postnikov section $P_nX$ is nilpotent. Assume that the map
$P_n\colon\text{map}_{\ast}(X,X)\to\text{map}_{\ast}(P_nX,P_nX)$ is a weak equivalence. Then
$p_{n,X}\colon X\to P_nX$ is a weak equivalence, i.e. $X$ is an $n$-Postnikov stage. \hfill{\qed}
\end{cor}

Here is another  application of Theorem~\ref{thm main} and  the characterization
of cellularity in Theorem~\ref{thm universal cell}. Notice however that the statement is not true for a general space~$X$.
Neither is the analogous statement  for higher connected covers of  a nilpotent $X$ as we have seen in
Example~\ref{ex connectedcover}.

\begin{cor}
\label{cor contractible}
Assume  $X$ is nilpotent. If  $\text{\rm map}_*(X,Y)$
is contractible, then so is $\text{\rm map}_*(P_nX, Y)$ for any $n\geq 0$. \hfill{\qed}
 \end{cor}

In the last part of this section we offer a few results which state simple - but not obvious - homological properties
of nilpotent spaces. All are immediate consequences of our main theorem.
\begin{thm}\label{thm homology}
Let $\kk$ be a reduced homology theory.
\begin{enumerate}
\item Assume $X$ is $\kk$-acyclic. If $P_nX$ is nilpotent, then it is also $\kk$-acyclic.
\item Assume  $X$ is nilpotent. Then $\prod_{k \geq 1} K(\pi_kX ,k)$ is $\kk$-acyclic if and only if
$\prod_{k \geq 1} K(H_kX ,k)$ is  $\kk$-acyclic.
\item If $K(G,1)$ is $\kk$-acyclic, then so is $K(G/\Gamma_n G,1)$ for any $n$. \hfill{\qed}
\end{enumerate}
\end{thm}


\bibliographystyle{plain}\label{biblography}

\end{document}